\documentclass[12pt]{amsart}

\usepackage{fullpage} 
\usepackage{parskip} 
\usepackage{tikz-cd} 
\usepackage{amsmath, amssymb, amsthm}
\usepackage{esint}
\usepackage{hyperref}
\hypersetup{
	colorlinks = false,
	linkbordercolor = {white},
	citebordercolor = {white},
}

\usepackage[utf8]{inputenc}
\usepackage[english]{babel}
\usepackage[alphabetic]{amsrefs}

\usepackage{graphicx}
\usepackage{float}
\usepackage{todonotes}

\usepackage[title]{appendix}

\usepackage{csquotes}

\usepackage{geometry}
\usepackage{fancyhdr}
\usepackage{lipsum}

\numberwithin{equation}{section}

\theoremstyle{definition}
\newtheorem{thm}[equation]{Theorem} 
\theoremstyle{definition}

\newtheorem{cor}[equation]{Corollary}
\newtheorem{lemma}[equation]{Lemma}
\newtheorem{remark}[equation]{Remark}

\newcommand{\pair}[1]{\left\langle #1\right\rangle}
\newcommand{\bb}[1]{\mathbb{#1}} 

\newcommand{\pr}[1]{\left(#1\right)}

\newcommand\numberthis{\addtocounter{equation}{1}\tag{\theequation}}

\newcommand{\D}{\Delta}
\newcommand{\Hess}{\text{Hess}}
\newcommand{\n}{\nabla}

\newcommand{\dl}{\mathcal L}
\newcommand{\sbst}{\subseteq}
\newcommand{\h}{H}

\newcommand{\R}{\textbf{R}}
\newcommand{\bd}{\partial}

\newcommand{\Ric}{\text{Ric}}
\newcommand{\DL}{\mathcal L}

\newcommand{\M}{\mathcal M}

\title{Parabolic frequency for the mean curvature flow}
\author{Julius Baldauf}

\author{Tang-Kai Lee}
\address{MIT, Department of Mathematics\\
77 Massachusetts Avenue, Cambridge, MA 02139.}
\email{juliusbl@mit.edu and tangkai@mit.edu}

\begin{document}
\begin{abstract}
This paper defines a parabolic frequency for solutions of the heat equation along homothetically shrinking mean curvature flows and proves its monotonicity along such flows. 
As a corollary, frequency monotonicity provides a proof of backwards uniqueness. 
Additionally, for solutions of more general parabolic equations on mean curvature flow shrinkers, this paper provides bounds on the derivative of the frequency, which similarly imply backwards uniqueness.
\end{abstract}
	\maketitle

\section{{\bf Introduction}}

To measure the growth rate of a harmonic function $u$ near a point $p$ in $\R^n$, Almgren \cite{A} introduced the frequency function 
\begin{equation}
    N(r)
    =\frac{r\|\nabla u\|_{L^2(B_r(p))}^2}{\|u\|_{L^2(\partial B_r(p))}^2},
\end{equation}
and proved its monotonicity. 
The frequency is a local measure of the ``degree'' of $u$ as a polynomial-like function in the ball $B_r(p)$. In particular, a harmonic function which is homogeneous of degree $k$ has constant frequency equal to $k$. 

Since its introduction by Almgren, the elliptic frequency has been useful in the study of nodal and critical sets of solutions to elliptic and parabolic equations and to prove unique continuation \cites{GL1, GL2, HL, HHL, Li, Lo}. 
Additionally, Colding-Minicozzi used frequency monotonicity to prove finite dimensionality of the space of polynomial growth harmonic functions on manifolds with nonnegative Ricci curvature and Euclidean volume growth \cite{CM1}. 
More recently, Taubes defined and studied versions of the elliptic frequency in the context of gauge theory \cite{T}. 
For applications such as these on a general manifold, an extension of Almgren's monotonicity due to Garofalo-Lin \cite{GL1} is necessary: there exists a constant $\Lambda>0$ such that $e^{\Lambda r^2}N(r)$ is monotone increasing for sufficiently small $r$; see also \cite{M}.


The parabolic frequency generalizes Almgren's original elliptic frequency for harmonic functions. On static manifolds, the parabolic frequency was first defined by Poon \cite{P}. 
Poon's frequency monotonicity theorem holds on the restricted class of manifolds satisfying the assumptions of Hamilton's matrix Harnack inequality, namely, manifolds with non-negative sectional curvature and parallel Ricci curvature \cites{H1,H2}. 
Using the drift Laplacian, \cite{CM21} circumvented using Hamilton's matrix Harnack inequality and proved monotonicity of the parabolic frequency for arbitrary static manifolds. 

This paper defines and studies the notion of parabolic frequency in the mean curvature flow (MCF). 
Parabolic frequency functions have previously been defined and studied in the context of manifolds evolving by Ricci flow \cite{LW, CM20, BK}. 
On flat Euclidean space, which can naturally be considered both as a MCF and as a Ricci flow, the definitions of parabolic frequency for Ricci flow and MCF both agree with Poon's.

The frequency function in the setting of MCF is defined as follows. 
Let $\Phi$ be the backward heat kernel on $\R^N$, centered at the origin, and let $\M$ be the spacetime track of a shrinking MCF $M_t=\sqrt{-t}M$ in $\R^N$, for $t<0$.
With $\h$ and $A$ denoting the mean curvature and second fundamental form of $M_t$, respectively, assume the curvature bound $\pair{\h,A}\le \frac{\kappa}{-t}$ holds along the flow.
Such a bound holds if $M_t$ is compact or asymptotically conical, for example.
The parabolic frequency of a solution $u\colon \M\to\bb R$ to the heat equation along the flow with\footnote{Some growth assumption is necessary to rule out the classical Tychonoff example \cite[Chapter 7]{J}.} 
$u,\partial_t u\in W^{2,2}(\mu_t)$ is defined to be
\begin{equation}\label{eqn: frequency defn}
    U(t) 
    := -
    \frac{2\int_{M_t} |\n u|^2 \Phi dV_{g_t}} {\int_{M_t} u^2 \Phi dV_{g_t}}(-t)^{1+2\kappa}
    = 
    \frac{2\int_{M_t} u\dl_t u\cdot \Phi dV_{g_t}} {\int_{M_t} u^2 \Phi dV_{g_t}}(-t)^{1+2\kappa},
\end{equation}
where $\dl_t$
is the drift Laplacian determined by the backward heat kernel measure.
The power involving $\kappa$ is a necessary correction term depending on the geometry of the flow via the aforementioned curvature bound; 
it is the parabolic analogue of the error term  $e^{\Lambda r^2}$ appearing in the above elliptic frequency.
By further analogy with the elliptic case, a caloric polynomial of degree $k$ has $U(t)\equiv -k$; see \cite{BK}.

\begin{thm}[Frequency monotonicity]\label{main}
 The frequency $U(t)$ is increasing, and $U'(t)=0$ only if $u$ is an eigenfunction of $\DL_t$ satisfying $\DL_tu=c(t)u$, where $c(t)=\frac{U(t)}{2(-t)^{1+2\kappa}}$.
\end{thm}


An important step in the proof of this theorem is the monotonicity formula for the weighted $L^2$-norm of solutions to the heat equation along MCF. The latter formula was used by Colding-Minicozzi \cites{CM20} to bound the codimension of a general MCF in terms of its entropy.
Frequency monotonicity can be applied to give a simple proof of backwards uniqueness.

\begin{cor}[Backward uniqueness]\label{bunique}
	If $a<b$ and $u(\cdot,b)\equiv 0,$ then $u(\cdot,t)\equiv 0$ for all $t\in[a,b]$.
\end{cor}

For solutions of more general parabolic equations along a MCF, though the frequency need not be monotone, its derivative can be bounded suitably to imply backwards uniqueness.

\begin{thm}[General backward uniqueness]\label{gmain}
 	Let $a<b$ and $u\colon \mathcal{M} \to \bb R$ satisfy $|(\bd_t-\D)u|\le C(t)\pr{|\n u|+|u|}$ with $\int_a^b C(t)^2 dt<\infty.$ If $u(\cdot,b)=0,$ then $u(\cdot,t)=0$ for all $t\in[a,b].$
\end{thm}

\begin{remark}[General type-I MCF]
The methods developed here can be extended to general MCFs with bounded curvature, such as type-I flows.
This extension requires a positive solution to the equation $\pr{\bd_t+\D-|\h|^2}K=0$, replacing the backward heat kernel \eqref{eq:phi}.
Such solutions exist; see \cite[Theorem 24.40]{RF}.
With this modification, the calculations to follow carry through as long as $\Hess_{\log K}$ is bounded, which is always the case if the flow is compact.
In general, the Hessian of the kernel can be bounded using the methods developed in \cites{HZ,H21}.
\end{remark}

The paper is organized as follows: Section \ref{sec: freq monotonicity} proves Theorem \ref{main} and Corollary \ref{bunique}; Section \ref{sec: general operator} proves Theorem \ref{gmain}; Section \ref{sec: consequences} discusses additional consequences and examples.

\subsection*{\bf Acknowledgements}
The authors are indebted to William Minicozzi for continual guidance and support and were partially funded by the NSF GRFP and NSF DMS Grant 2005345.

\section{\bf Frequency Monotonicity}\label{sec: freq monotonicity}
Let $M$ be an $n$-dimensional self-shrinker in $\bb R^{N},$ and $M_t^n$ ($t\in (-\infty,0]$) be the MCF induced by it. 
That is, $M$ satisfies $\h = -\frac{x^\perp}{2}$ where $x^\perp$ is the normal component of the position vector. 
By a change-of-variable argument, its rescaling $M_t:=\sqrt{-t}M$ satisfies the MCF equation $\bd_t x = \h(x,t).$
Assume $\pair{\h,A}\le \kappa$ on $M$ for some $\kappa>0,$ which, by scaling, implies
\begin{equation}\label{HAbound}
\pair{\h,A}\le \frac{\kappa}{-t}
\end{equation}
on $M_t$ for any $t<0.$

To define the energy and the frequency for functions along $M_t,$ let $\Phi$ be the $n$-dimensional backward heat kernel. 
That is,
\begin{equation}\label{eq:phi}
    \Phi(x,t):=
\frac 1{(-4\pi t)^{n/2}} e^{\frac{|x|^2}{4t}}
\end{equation}
for $(x,t)\in\bb R^N\times (-\infty,0].$
Using this, consider the measure
$$d\mu_t(x) := \Phi(x,t) dV_{g_t}(x),$$
where $dV_{g_t}$ is the induced volume form on $M_t.$  
We let $\M_{[a,b]}:=\bigcup_{a\le t\le b} (M_t\times\{t\})$ be the spacetime track of the flow on the time interval $[a,b]\sbst (-\infty,0).$ 
Let $u\colon \M_{[a,b]}\to\bb R$ be a $C^2$ function such that $u,\bd_t u\in W^{2,2}(\mu_t)$ when restricted to each time slice. Given such $u,$ define
\begin{equation*}
I(t):=
\int_{M_t} u^2 d\mu_t,
\end{equation*} 
\begin{equation*}
D(t):=
-2\int_{M_t} |\n u|^2 d\mu_t
= 2\int_{M_t}u\dl_t u\cdot d\mu_t,
\end{equation*} 
and
\begin{equation*}
U(t) := (-t)^{1+2\kappa}\frac{D(t)}{I(t)}
\end{equation*}
where $\dl_t:=e^{-\frac{|x|^2}{4t}}\text{div}\pr{e^{\frac{|x|^2}{4t}}\n(\cdot)}$ is the drift Laplacian determined by the measure $d\mu_t.$

The first result that we will keep using when calculating the derivatives of $I$ and $D$ is the following weighted monotonicity formula, which holds for general MCF in any dimension and codimension, and was first proven in \cite{H90} and \cite{E}.

\begin{thm}[{\cite[Theorem 4.13]{E}}]\label{mono}
	Suppose $f$ is a smooth function on $\M_{[a,b]}$ with $f,\bd_t f\in W^{2,2}(\mu_t)$ when restricted to each time slice. 
	Then
	$$\frac{d}{dt}
	\int_{M_t} fd\mu_t
	= \int_{M_t} \left((\bd_t-\D)f - 
	\left|\h-\frac{x^\perp}{2t}\right|^2f
	\right)d\mu_t.
	$$
\end{thm}

The second calculation that we will need is the following lemma.

\begin{lemma}\label{driftB}
	If $u\in W^{2,2}(\mu_t)$ when restricted to each time slice, then
	$$\int_{M_t} \pr{|\Hess_u|^2 +  \Ric(\n u,\n u)} d\mu_t
	= \int_{M_t} \pr{
	(\dl_t u)^2
	- \frac 1{-2t}|\n u|^2
	}d\mu_t.
	$$
\end{lemma}

\begin{proof}
	We have the drift Bochner formula
	$$
	\frac 12 \DL_f |\n u|^2
	= |\Hess_u|^2 
	+ \Ric(\n u,\n u)
	+ \pair{\n u,\n \DL_f u}
	+ \Hess_f(\n u,\n u)
	$$
	for any smooth function $f$ and the associated drift Laplacian $\dl_f:=e^f\text{div}\pr{e^{-f}\n(\cdot)}.$
	Then since $\dl_t$ is the drift Laplacian determined by $d\mu_t,$ using integration by parts, we have
	\begin{align*}
	0 
	& = \frac 12 \int_{M_t} \pr{\dl_t|\n u|^2} d\mu_t\\
	& = \int_{M_t} \pr{|\Hess_u|^2 
		+ \Ric(\n u,\n u)
		+ \pair{\n u,\n \DL_t u}
		+ \Hess_{{-\frac{|x|^2}{4t}}}(\n u,\n u)}d\mu_t\\
	& = \int_{M_t} \pr{|\Hess_u|^2 
		+ \Ric(\n u,\n u)
		- \pr{\dl_tu}^2
		+ \frac 1{-2t}|\n u|^2}d\mu_t
	\end{align*}
	since $\Hess_{{-\frac{|x|^2}{4t}}} = \frac{1}{-2t}\pair{\cdot,\cdot}_{\text{Euc}}.$ Then the lemma follows.
\end{proof}

Using this lemma and the boundedness assumption on the curvatures, we can prove the first main theorem.

\begin{proof}[Proof of Theorem \ref{main}]
Suppose $u$ is a solution to the heat equation. i.e., $(\bd_t-\D)u=0.$ Since
$$(\bd_t-\D)u^2=2u(\bd_t-\D)u-2|\n u|^2=-2|\n u|^2$$
and the self-shrinking MCF satisfies $\h=\frac{x^\perp}{2t},$ Theorem \ref{mono} implies
\begin{equation}\label{eq:I'}
    I'(t)
= \int_{M_t} -2|\n u|^2 d\mu_t
= D(t).
\end{equation}

To get $D'(t),$ we calculate $(\bd_t-\D)|\n u|^2.$ 
Along the MCF, we have (cf. \cite{H})
$$\bd_t g_{ij} = -2 \pair{\h, A_{ij}},$$
which implies
\begin{equation}\label{gt}
\bd_t g^{ij} = 2 g^{ik}g^{jl}\pair{\h, A_{kl}}.
\end{equation}
Therefore, working at a point $p\in M_t$ and fixing an orthonormal tangent frame $e_1,\dots,e_n$ near $p,$ we can calculate
\begin{align*}
\bd_t|\n u|^2
& = \bd_t(g^{ij} \n_iu \n_j u)\\
& = 2g^{ik}g^{jl}\pair{\h, A_{kl}  \n_iu \n_j u}
+ 2 g^{ij} \n_i (\bd_t u)\n_ju\\
& = 2\pair{\h,A(\n u,\n u)} + 2\pair{\n \D u,\n u}\\
& = 2\pair{\h,A(\n u,\n u)} 
+ \D |\n u|^2
- 2|\Hess_u|^2 
- 2\Ric(\n u,\n u)
\end{align*}
using the standard Bochner formula. 
Thus, by Theorem \ref{mono} and Lemma \ref{driftB}, we obtain
\begin{align*}
D'(t)
& = -2 \int_{M_t} (\bd_t-\D)|\n u|^2 d\mu_t
\numberthis\label{eq:D'}\\
& = -2 \int_{M_t} \pr{2\pair{\h,A(\n u,\n u)} 
	- 2|\Hess_u|^2 
	- 2\Ric(\n u,\n u)} d\mu_t\\
& = -4\int_{M_t} \pr{
\pair{\h, A(\n u,\n u)}
- (\dl_t u)^2
+ \frac 1{-2t}|\n u|^2
}d\mu_t\\
& \ge \int_{M_t} \left(
4(\DL_t u)^2 -\pr{\frac 2{-t}+\frac{4\kappa}{-t}}|\n u|^2
\right)d\mu_t\\
& = 4\int_{M_t} (\DL_t u)^2 d\mu_t + \frac{1+2\kappa}{-t}D(t)
\end{align*}
where we use assumption \eqref{HAbound} (which implies $\pair{\h,A(\n u,\n u)}\le \frac{\kappa}{-t}|\n u|^2$).

Combining \eqref{eq:I'} and \eqref{eq:D'}, we derive
\begin{align*}
I(t)^2 U'(t)
& = (-t)^{1+2\kappa} (I(t)D'(t)-I'(t)D(t)) + \left(- (1+2\kappa) (-t)^{2\kappa}\right)I(t) D(t) \\
& \ge (-t)^{1+2\kappa} \left(
4 I(t)\int_{M_t} (\DL_t u)^2 d\mu_t
+ \frac{1+2\kappa}{-t} I(t) D(t)
- D(t)^2  
- \frac{1+2\kappa}{-t} I(t) D(t)
\right)\\
& = (-t)^{1+2\kappa} \left(
4\int_{M_t}u^2d\mu_t\cdot \int_{M_t} (\DL_t u)^2 d\mu_t
- 4\left(
\int_{M_t} u\DL_t u \cdot d\mu_t
\right)^2
\right),
\end{align*}
which is non-negative by the Cauchy-Schwarz inequality. As a consequence, $U$ is increasing.

If $U'(t)=0$, then equality in the Cauchy-Schwarz inequality implies $\DL_t u = c(t)u$. The function $c(t)$ can be determined by noting that if $\DL_tu=c(t)u$, then
\begin{align*}
    I(t)U(t)
    &=(-t)^{1+2\kappa}D(t) \\
    &= 2(-t)^{1+2\kappa}\int_M u\DL_tu \cdot \,d\mu_t \\
    &=2 (-t)^{1+2\kappa} c(t)\int_Mu^2 \,d\mu_t \\
    &=2(-t)^{1+2\kappa} c(t)I(t),
\end{align*}
which implies that $c(t)=\frac{U(t)}{2(-t)^{1+2\kappa}}$.
\end{proof}

The frequency monotonicity can be used to derive a Harnack-type inequality for $I$, which immediately leads to backward uniqueness.

\begin{cor}[Precise version of Corollary \ref{bunique}]\label{cor:Harnack}
	Suppose $u\colon \M_{[a,b]}\to\bb R$ satisfies $u,\bd_t u\in W^{2,2}(\mu_t)$ and $(\bd_t-\D)u=0.$ If $\kappa>0,$ then
	\begin{equation}\label{eq:k>0}
	    I(b) \ge I(a) \cdot e^{\frac 1{2\kappa}\left(
		(-b)^{-2\kappa}-(-a)^{-2\kappa}
		\right)
		U(a)};
	\end{equation}
	if $\kappa=0,$ then
	\begin{equation}\label{eq:k=0}
	    I(b)\ge I(a)\cdot e^{-U(a)}\pr{\frac ba}.
	\end{equation}
	In particular, if $u(\cdot,b)=0,$ then $u(\cdot,t)=0$ for all $t\in[a,b].$
\end{cor}
\begin{proof}
Suppose $\kappa>0.$ Since
\begin{equation}\label{logI'}
(\log I)'(t) = \frac {D(t)}{I(t)} = (-t)^{-1-2\kappa} U(t),
\end{equation}
for $[a,b]\sbst (-\infty,0),$
\begin{align}\label{eqn: difference of log(I)}
\log I(b)-\log I(a)
&= \int_a^b (-t)^{-1-2\kappa} U(t) dt
\\\nonumber
&\ge U(a) \int_a^b (-t)^{-1-2\kappa} dt
\\\nonumber
&= \frac 1{2\kappa}\left(
(-b)^{-2\kappa}-(-a)^{-2\kappa}
\right)
U(a) 
\end{align}
Thus,
\begin{equation}\label{IHarnack}
I(b) \ge I(a) \cdot e^{\frac 1{2\kappa}\left(
	(-b)^{-2\kappa}-(-a)^{-2\kappa}
	\right)
	U(a)}.
\end{equation}

When $\kappa=0,$ \eqref{logI'} implies $(\log I)'(t)=(-t)^{-1}U(t).$ Therefore, integrating over $[a,b]$ gives
\begin{align*}
\log I(b)-\log I(a)
& = \int_a^b (-t)^{-1}U(t) dt\\
& \ge U(a) \int_a^b (-t)^{-1}dt\\
& = -U(a)\log\pr{\frac ba},
\end{align*}
so
$$I(b)\ge I(a)\cdot e^{-U(a)} \pr{\frac ba},$$
and the conclusion follows. 
\end{proof}

As a consequence of this Harnack-type inequality \eqref{eq:k>0}, we can see that a solution to the heat equation cannot decay too fast when it approaches the singularity unless it is constant. 
	
\begin{cor}
	Suppose $u\colon \M_{[a,0)}\to\bb R$ satisfies $u,\bd_t u\in W^{2,2}(\mu_t)$ and $(\bd_t-\D)u=0.$ If $\lim\limits_{t\to 0^-} e^{c(-t)^{-2\kappa}}I(t)=0$ for any $c>0,$ then $I(t)=0$ for all $t,$ which means $u$ is constant.
\end{cor}

This property implies that the frequency function leads to a restriction of the growth rate of solutions to the heat equation. 
In fact, when $M_t=\bb R^n$ is a static solution to the MCF, Poon \cites{P} used \eqref{eq:k=0} to show that the solution to the heat equation cannot vanish of infinite order unless it is constant.

\section{\bf More General Operators}\label{sec: general operator}
In this section, we deal with the case when $u$ is not an exact solution to the heat equation. 
As before, we assume $M_t = \sqrt{-t} M$ is a shrinking MCF with $\pair{\h, A} \le \frac{\kappa}{-t}$ and $\M_{[a,b]}$ is the space-time track of the flow on the time interval $[a,b]\sbst (-\infty,0).$ 
\begin{thm}
	Suppose $u\colon\M_{[a,b]}\to\bb R$ satisfies $u,\bd_t u\in W^{2,2}(\mu_t)$ and 
	\begin{equation}\label{generalassumption}
	|(\bd_t-\D)u|\le C(t)\pr{|\n u|+|u|}
	\end{equation}
	for some time-dependent constant $C(t).$
	Then we have
	\begin{equation}\label{Ibound}
	(\log I(t))'\ge \pr{1+\frac{C(t)}2} (-t)^{-1-2\kappa} U(t) - 3C(t)
	\end{equation}
	and
	\begin{equation}\label{Ubound}
	U'(t)
	\ge C(t)^2\pr{U(t) - 2(-t)^{1+2\kappa}}
	\end{equation}
	for any $t<0.$
\end{thm}

\begin{proof}
	Using Theorem \ref{mono} and assumption \eqref{generalassumption}, we have
	\begin{align*}
	I'(t)
	= \int_{M_t} (\bd_t-\D)u^2 d\mu_t
	& = \int_{M_t} \pr{2u(\bd_t-\D)u - 2|\n u|^2} d\mu_t\\
	& = D(t) + 2\int_{M_t} u(\bd_t-\D)u\cdot d\mu_t\\
	& \ge D(t) - 2C(t) \int_{M_t} |u|\pr{|\n u|+|u|} d\mu_t\\
	& = D(t) - 2C(t)I(t) - 2C(t)\int_{M_t} |\n u|\cdot |u| d\mu_t\\
	& \ge D(t) - 2C(t)I(t) - C(t)\pr{I(t)-\frac 12 D(t)}\\
	& = \pr{1+\frac{C(t)}2} D(t) - 3C(t)I(t)\\
	& = \pr{1+\frac{C(t)}2} (-t)^{-1-2\kappa}I(t) U(t) - 3C(t)I(t)
	\end{align*}
	where we use the inequality $2ab\le a^2+b^2.$ Therefore, \eqref{Ibound} follows.
	
	To bound $U'(t)$ in this general situation, we first rewrite 
	\begin{align*}
	D(t)
	= -2\int_{M_t} |\n u|^2 d\mu_t
	& = 2\int_{M_t} u\dl_t u\cdot d\mu_t\\
	& = 2\int_{M_t} u\pr{\dl_t u+\frac 12(\bd_t-\D)u}d\mu_t 
	- \int_{M_t} u(\bd_t-\D)u\cdot d\mu_t
	\end{align*}
	and
	\begin{align*}
	I'(t)
	& = \int_{M_t} \pr{2u(\bd_t-\D)u + 2u\dl_t u} d\mu_t\\
	& = 2\int_{M_t} u\pr{\dl_t u+\frac 12(\bd_t-\D)u}d\mu_t 
	+ \int_{M_t} u(\bd_t-\D)u\cdot d\mu_t.
	\end{align*}
	Hence,
	\begin{equation}\label{I'D}
	I'(t) D(t)
	= \pr{\int_{M_t} u\pr{2\dl_t u
			+(\bd_t-\D)u}d\mu_t}^2
	- \pr{\int_{M_t} u(\bd_t-\D)u\cdot d\mu_t}^2.
	\end{equation}
	
	For $D'(t),$ note that \eqref{gt} implies
	\begin{align*}
	\bd_t|\n u|^2
	& = 2\pair{\h,A(\n u,\n u)} + 2\pair{\n (\bd_t u),\n u}\\
	& = 2\pair{\h,A(\n u,\n u)} 
	+ 2\pair{\n \D u,\n u} 
	+ 2\pair{\n(\bd_t-\D)u,\n u}\\
	& = 2\pair{\h,A(\n u,\n u)} 
	+ \D |\n u|^2
	- 2|\Hess_u|^2 
	- 2\Ric(\n u,\n u)
	+ 2\pair{\n(\bd_t-\D)u,\n u}
	\end{align*}
	by the standard Bochner formula. Thus, Theorem \ref{mono}, Lemma \ref{driftB}, and assumption \eqref{HAbound} again imply
	\begin{align*}
	D'(t)
	& = -2 \int_{M_t} (\bd_t-\D)|\n u|^2 d\mu_t\\
	& = -2 \int_{M_t} \pr{2\pair{\h,A(\n u,\n u)} 
	- 2|\Hess_u|^2 
	- 2\Ric(\n u,\n u)
	+ 2\pair{\n(\bd_t-\D)u,\n u}} d\mu_t\\
	& = -4\int_{M_t} \pr{
		\pair{\h, A(\n u,\n u)}
		- (\dl_t u)^2
		+ \frac 1{-2t}|\n u|^2
		- (\bd_t-\D)u\cdot\dl_t u
		}d\mu_t\\
	& \ge -\frac{2+4\kappa}{-t}\int_{M_t} |\n u|^2 d\mu_t 
	+ \int_{M_t} \pr{4(\DL_t u)^2 
	+ 4(\bd_t-\D)u\cdot\dl_t u} d\mu_t\\
	& = \frac{1+2\kappa}{-t} D(t) 
	+ \int_{M_t} \pr{
	\pr{2\dl_tu + (\bd_t-\D)u}^2 - \pr{(\bd_t-\D)u}^2
	}d\mu_t.
	\end{align*}
	Combining this with \eqref{I'D}, we get
	\begin{align*}
	I(t)^2 U'(t)
	& = (-t)^{1+2\kappa} (I(t)D'(t)-I'(t)D(t)) + \left(- (1+2\kappa) (-t)^{2\kappa}\right)I(t) D(t) \\
	& \ge (-t)^{1+2\kappa} \pr{
	\frac{1+2\kappa}{-t} I(t) D(t) 
	+ \pr{\int_{M_t}u^2d\mu_t} 
	\pr{ \int_{M_t} \pr{
			\pr{2\dl_tu + (\bd_t-\D)u}^2 - \pr{(\bd_t-\D)u}^2
		}d\mu_t}}\\
	& - (-t)^{1+2\kappa} \pr{\pr{\int_{M_t} u\pr{2\dl_t u
				+ (\bd_t-\D)u}d\mu_t}^2
	- \pr{\int_{M_t} u(\bd_t-\D)u\cdot d\mu_t}^2
	}\\
	& - (-t)^{1+2\kappa} \frac{1+2\kappa}{-t} I(t) D(t) \\
	& \ge - (-t)^{1+2\kappa} I(t) \int_{M_t} \pr{(\bd_t-\D)u}^2 d\mu_t
	\end{align*}
	using the H\"older inequality. As a result, plugging in the assumption \eqref{generalassumption}, we obtain
	\begin{align*}
	I(t)^2 U'(t)
	& \ge - (-t)^{1+2\kappa} I(t) \cdot C(t)^2 \int_{M_t} \pr{|\n u|+|u|}^2 d\mu_t \\
	& \ge (-t)^{1+2\kappa} I(t) \cdot C(t)^2 \pr{D(t)-2I(t)}
	\end{align*}
	by the inequality $(a+b)^2\le 2(a^2+b^2).$ This then implies
	\begin{align*}
	U'(t)
	\ge (-t)^{1+2\kappa} C(t)^2 \pr{\frac{D(t)}{I(t)}-2}
	= C(t)^2\pr{U(t) - 2(-t)^{1+2\kappa}}
	\end{align*}
	and \eqref{Ubound} follows.
\end{proof}

\begin{cor}[Precise version of Theorem \ref{gmain}]
	If $u\colon\M_{[a,b]}\to\bb R$ satisfies $u,\bd_t u\in W^{2,2}(\mu_t)$ and \eqref{generalassumption}, then we have
	\begin{align*}
	I(b)
	\ge I(a) \exp\biggl(\int_a^b \pr{1+\frac{C(t)}{2}} (-t)^{-1-2\kappa} &\left(\pr{U(a)-2(-a)^{1+2\kappa}} e^{\int_a^t C(s)^2 ds}\right.
	\\ 
	& \left. + 2(-a)^{1+2\kappa}\right) dt 
	- 3\int_a^b C(t)dt\biggr).
	\end{align*}
	In particular, if $\int_a^b C(t)^2 dt<\infty$ and $u(\cdot,b)=0,$ then $u(\cdot,t)=0$ for all $t\in[a,b].$
\end{cor}
\begin{proof}
	Using \eqref{Ubound}, we have
	\begin{align*}
	\pr{\log\pr{2(-a)^{1+2\kappa}-U(t)}}'
	= \frac {-U'(t)} {2(-a)^{1+2\kappa}-U(t)}
	\le \frac {C(t)^2\pr{2(-t)^{1+2\kappa}-U(t)}}  {2(-a)^{1+2\kappa}-U(t)}
	\le C(t)^2 
	\end{align*}
	for any $t\in[a,b].$ After integration, this implies
	\begin{align*}
	\log\pr{2(-a)^{1+2\kappa}-U(t)} 
	\le \log\pr{2(-a)^{1+2\kappa}-U(a)}
	+ \int_a^t C(s)^2 ds,
	\end{align*}
	so
	$$
	U(t)\ge \pr{U(a)-2(-a)^{1+2\kappa}} e^{\int_a^t C(s)^2 ds} + 2(-a)^{1+2\kappa}.
	$$
	Using this and \eqref{Ibound}, we have
	\begin{align*}
	&~~~~\log I(b)-\log I(a)\\
	& \ge \int_a^b \pr{1+\frac{C(t)}2} (-t)^{-1-2\kappa} U(t) dt 
	- 3\int_a^b C(t)dt\\
	& \ge \int_a^b \pr{1+\frac{C(t)}2} (-t)^{-1-2\kappa} \pr{\pr{U(a)-2(-a)^{1+2\kappa}} e^{\int_a^t C(s)^2 ds} + 2(-a)^{1+2\kappa}} dt 
	- 3\int_a^b C(t)dt.
	\end{align*}
	Then the conclusion follows after exponentiating.
\end{proof}

\section{\bf Other Consequences and Examples}\label{sec: consequences}
\subsection{Eigenvalue Monotonicity}
Let $\lambda_1(t)$ be the first eigenvalue of the drift Laplacian $\dl_t=\dl_{-\frac{|x|^2}{4t}}.$ 
We know that it is given by
\begin{equation}\label{lambda1}
\lambda_1(t)
= \inf\left\{
\frac{\int_{M_t} |\n f|^2 d\mu_t}{\int_{M_t} f^2 d\mu_t}: f\in W_0^{1,2}(M_t)\setminus\{0\}
\right\}.
\end{equation}

Given $[a,b]\sbst (-\infty,0),$ let $u$ be the solution to the heat equation along the flow $M_t$ with $\dl_a u(\cdot,a)=\lambda_1(a) u(\cdot,a).$ 
That is, $u(\cdot,a)$ is the first eigenfunction of $\dl_a$ on $M_a.$ 
Define the frequency function $U$ for $u$ by
$$U(t) 
:= (-t)^{1+2\kappa}
\frac{-2\int_{M_t}|\n u|^2 d\mu_t}{\int_{M_t}u^2d\mu_t}.$$
Theorem \ref{main} implies that $U(t)$ is increasing, which tells us that for any $t\in[a,b],$ we have
$$
-\frac 12 (-a)^{1+2\kappa}\lambda_1(a)
= U(a)
\le U(t)
\le -\frac 12 (-t)^{1+2\kappa}\lambda(t)
$$
based on \eqref{lambda1}. 
As a result, we see that $(-t)^{1+2\kappa}\lambda(t)$ is decreasing. 

Note that although we know the exact behavior of $\lambda_1(t)$ when $M_t$ is a shrinking MCF, it is not the case when $M_t$ is a general solution to the MCF. 
In this general situation, the monotonicity derived here is also valid. 

\subsection{Examples}	
In \cite{BK}, using the spectrum of the shrinking Gaussian drift Laplacian, the authors showed that a function $u\colon \bb R^n\times (-\infty,0)\to\bb R$ has constant frequency if and only if $u$ is a caloric polynomial. 
That also fits into our framework when $M_t=\bb R^n$ is a static solution to the MCF. 
Here, we further analyze the case of shrinking spheres.

Consider $M_t := S^n_{\sqrt{-2nt}}\sbst\bb R^{N},$ a shrinking solution to the MCF, where $S^n_r$ means the sphere centered at the origin with radius $r.$ 
If $u\colon \M_{[a,b]}\to\bb R$ is a solution to the heat equation along $M_t$ on $[a,b]$ with $U(t)=U(b)$ for all $t\in[a,b],$ then by the equality case of Theorem \ref{main}, we know that 
$$\dl_t u 
= c(t) u$$
with $c(t)=\frac{U(b)}{2(-t)^{1 + 2\kappa}}.$ 
In the case of spheres, we have that the drift term vanishes in the $\dl$ operator. 
That is,
$$\dl_t 
= \dl_{\frac{|x|^2}{-4t}}
= \D_{M_t} - \frac{1}{-2t}\n_{x^T} = \D_{M_t}.$$
Therefore, we know that $u$ is a spherical harmonic at each time slice with 
$$c(-1)
= -\frac{k^2 + (n-1)k}{2n}$$
for some $k\in\bb N\cup\{0\}$ if we assume, without loss of generality, that $-1\in[a,b].$ 
In particular, we know that 
$$U = -\frac 1n \pr{k^2 + (n-1)k}.$$
Moreover, by the backward uniqueness (Corollary \ref{bunique}) and \cite[Lemma 2.4]{CM20}, we can see that 
$$u(x,t)
= (-t)^{-\frac{U}{2}} u\pr{\frac{x}{\sqrt{-t}},-1}.$$
That is, $u$ is explicitly determined by its behavior at any fixed time slice.

    

    
    
    
    


	
\end{document}